\documentclass{amsart}
\usepackage{enumerate}
\usepackage{amsmath}
\usepackage{mathrsfs}
\usepackage{color}
\usepackage{amssymb}

\usepackage[colorlinks=true,urlcolor=blue, citecolor=blue,linkcolor=blue,linktocpage,pdfpagelabels,bookmarksnumbered,bookmarksopen]{hyperref}

\numberwithin{equation}{section}

\newtheorem{theorem}{Theorem}[section]
\newtheorem{lemma}[theorem]{Lemma}
\newtheorem{corollary}[theorem]{Corollary}
\newtheorem{remark}[theorem]{Remark}
\newtheorem{definition}[theorem]{Definition}

\newtheorem{proposition}[theorem]{Proposition}

\newtheorem{assumption}[theorem]{Assumption}

\begin{document}
	\title[Noncommutative Stein's maximal spherical means]{Noncommutative Stein's maximal spherical means}
	\author[Li]{Wei Li}
	\author[Li]{Wenjuan Li}
	\author[Liu]{Jie Liu}
	\author[Wu]{Lian Wu}
	\address{Wei Li, School of Mathematics and Statistics, Central South University, Changsha 410075, China}
	\email{liwei.csu@csu.edu.cn}
	\address{Wenjuan Li, School of Mathematics and Statistics, Northwestern Polytechnical University, Xi'an 710129, China}
	\email{liwj@nwpu.edu.cn}
	\address{Jie Liu, School of Mathematics and Statistics, Northwestern Polytechnical University, Xi'an 710129, China}
	\email{jay2000@mail.nwpu.edu.cn}
	\address{Lian Wu, School of Mathematics and Statistics, Central South University, Changsha 410075, China}
	\email{wulian@csu.edu.cn}

	
	\thanks{This work is supported by the National Natural Science Foundation of China (No.12271435, No.11971484).}
	
	\begin{abstract}
		Based on a proper hypothesis on the noncommutative Fourier integral operators, we establish in this paper the  strong-type $(p,p)$ (with $2\leq p\leq \infty$) estimates
		for the operator-valued Stein's maximal spherical means.
	\end{abstract}
	
	\maketitle
	\section{introduction}
	The aim of this paper is to study the boundedness of 
	Stein's maximal spherical means acting on noncommutative $L_{p}$-spaces. These means 
	were firstly introduced by Stein \cite{S1976} in 1976 and have been found significant importance to harmonic analysis and  partial differential equations. For $f \in C_0^{\infty}(\mathbb{R}^n)$, the classical maximal spherical means of (complex) order $\alpha$ of $f$ are defined as follows:
	\begin{equation}\label{1.1}
		\mathscr{M}^\alpha f
		=
		\sup_{t>0}\left|\mathscr{M}_t^\alpha f(x)\right|
		=
		\sup_{t>0}\left| \left(f * m_{\alpha, t}\right)(x)\right|,
	\end{equation}
	where
	$
	m_{\alpha, t}(x)=m_\alpha(x / t) t^{-n}
	$ for $t>0$ and
	$
	m_\alpha(x)=\Gamma(\alpha)^{-1}\left(1-|x|^2\right)_{+}^{\alpha-1}.
	$
	Here $\Gamma(\alpha)$ denotes the Gamma function and $r_+$ is the homogeneous distribution which equals to $r$ when $r>0$, and equals to $0$ if $r\leq 0$.
	These averaging operators  not only have definitions for real positive $\alpha$, but can also be similarly defined for all complex $\alpha$. Indeed,  consider the Fourier transform of $m_\alpha$,
	\begin{equation}\label{1.2}
		\widehat{m_\alpha}(\xi)=\pi^{-\alpha+1}|\xi|^{-\frac{n}{2}-\alpha+1} \mathcal{J}_{\frac{n}{2}+\alpha-1}(2 \pi|\xi|),
	\end{equation}
	where $\mathcal{J}_m(x)$ is the Bessel function of order $m$ (see \cite{SW1971} or the Appendix of \cite{MYZ2017}).  Notice that, by means of analytic continuation, $\operatorname{Re} \alpha \leq 0$ is allowed in the Bessel functions in (\ref{1.2}). It is therefore possible to extend the notion of (\ref{1.1}) to include complex $\alpha$ via Fourier transform
	$$
	\widehat{\mathscr{M}_t^\alpha f}(\xi)=\widehat{m_\alpha}(\xi t) \widehat{f}(\xi).
	$$
	It is also important to remark that $\widehat{m}_\alpha(0)$ is finite and $\widehat{m}_\alpha(\xi)$ is smooth near the origin. For the detailed discussion  on these standard facts, we refer to \cite{S1976} and \cite[Theorem 4.15]{SW1971}.

	In \cite[Theorem 2]{S1976},  Stein proved that when $n\geq 3$,
	\begin{equation}\label{boundedness}
		\left\|\mathscr{M}^\alpha f\right\|_{L_p\left(\mathbb{R}^n\right)}
		\lesssim \|f\|_{L_p\left(\mathbb{R}^n\right)}
	\end{equation}
	if
	$$
	\operatorname{Re} \alpha>1-n+\frac{n}{p} \quad \mbox{for}\,\, 1<p \leq 2,
	$$
	or
	\begin{equation}\label{rangep}
		\operatorname{Re} \alpha>\frac{2-n}{p} \quad \mbox{for}\,\, 2 \leq p \leq \infty.
	\end{equation}
	In the special case  $\alpha=0$, the maximal spherical operator $\mathscr{M}^0$ is bounded on $L_p(\mathbb{R}^n)$ for $p>n /(n-1)$ when $n \geq 3$. This result is sharp in the sense that such boundedness fails for $p \leq n /(n-1)$ if $n \geq 2$. When $n=2$, Bourgain \cite{B1986} proved that for some positive constant $\varepsilon(p)$ the following holds:
	$$
	\left\|\mathscr{M}^\alpha f\right\|_{L_p\left(\mathbb{R}^2\right)}
	\lesssim \|f\|_{L_p\left(\mathbb{R}^2\right)}, \quad \operatorname{Re} \alpha>-\varepsilon(p), \quad 2<p<\infty .
	$$
	Subsequently, based on the local smoothing estimate for the solutions of wave equations, Mockenhaupt, Seeger and Sogge \cite{MSS1992}  strengthened  Bourgain's result. Their work reveals for the first time the essential connection  between the maximal circular means and the local smoothing estimates.

	Later, using the Bourgain-Demeter decoupling theorem \cite{BD2015} and the ideas discussed in \cite{MSS1992}, Miao, Yang and Zheng \cite[Theorem 1.1]{MYZ2017} obtained certain improvements on the range of $\alpha$ in (\ref{boundedness}). Their result can be stated as follows: for $n \geq 2$ and $p \geq 2$, the estimate (\ref{boundedness}) holds whenever
	$$
	\operatorname{Re} \alpha>\max \left\{\frac{1-n}{4}+\frac{3-n}{2 p}, \frac{1-n}{p}\right\} \text {. }
	$$
	However, this range is still not optimal, even though it is much better than that of (\ref{rangep}). Recently, in dimension two, Liu, Shen, Song and Yan \cite[Theorem 1.1]{LSSY2023} further improved the range of $\alpha$ to the following
	$$
	\operatorname{Re} \alpha > \max \left\{\frac{1}{p}-\frac{1}{2},-\frac{1}{p}\right\},
	$$
	which is pretty much close to the sharpest.

	The purpose of this paper is to provide some extension of the above results in the noncommutative setting. Before presenting the results, we briefly review some related developments on noncommutative maximal inequalities. Due to the lack of maximal functions, noncommutative maximal inequalities have been  mysterious  for quite a long time until the appearance of Pisier's vector-valued noncommutative $L_p$-spaces (see \cite{P1998}). The core idea of Pisier is that instead of noncommutative maximal functions, one could define the $L_p(\mathcal{M};\ell_{\infty})$-norm of the sequence of operators $(f_n)_n$,  which exactly corresponds to the $L_p$-norm of the maximal function $\sup_{n}|f_n|$ in the commutative case. This definition was later extended to more general case by Junge \cite{J2002} in 2002, which makes  noncommutative Doob inequalities \cite{J2002}, noncommutative Hardy-littlewood maximal inequalities \cite{M2007}, noncommutative maximal inequalities for ergodic averages \cite{JX2007} and truncated operator-valued singular integrals  \cite{HLX2022} possible. For noncommutative generalizations of other classical results, we refer to \cite{CXY2013, GJP2017, GJP, HWW, JM2010, JM2012, JMP2014, JMP2018, JMPX2021, PRS, XXX2016, XXY2018, Y1977}.
	
	Let $\mathcal{M}$ be a von Neumann algebra equipped with a normal semifinite faithful trace $\tau$. Denote by $\mathcal{N}=L_{\infty}\left(\mathbb{R}^n\right) \overline{\otimes} \mathcal{M}$ the tensor von Neumann algebra equipped with the tensor trace $\int_{\mathbb{R}^n} d x \otimes \tau$. Let $L_p(\mathcal{M})$, $L_p(\mathcal{N})$ be the noncommutative $L_p$ spaces associated to the pairs $(\mathcal{M}, \tau)$ and $(\mathcal{N}, \int_{\mathbb{R}^n} d x \otimes \tau)$, respectively. 
	We consider the noncommutative Stein's spherical means defined as
	$$
	\widehat{\mathscr{M}_{t}^{\alpha}f}(\xi)
	=
	\widehat{m_{\alpha}}(\xi t)\widehat{f}(\xi), \quad f\in \mathcal{S}(\mathbb{R}^n;L_p(\mathcal{M})).
	$$
	Here, $\mathcal{S}(\mathbb{R}^n;L_p(\mathcal{M}))$ stands for the class of $L_p(\mathcal{M})$-valued Schwartz functions on $\mathbb{R}^n$.
	Motivated by  \cite{MYZ2017, LSSY2023},
	we give bounded results for noncommutative Stein's maximal spherical means based on the following assumption on noncommutative Fourier integral operators.
	\begin{assumption}\label{assum}
		Let $1<p<\infty$, $s_p=(n-1)\big|\frac{1}{2}-\frac{1}{p}\big|$ and
		$$
		\bar{p}_n:=
		\begin{cases}
			\frac{2(n+1)}{n-1}, & \text { if } n \mbox { is odd,} \\
			\frac{2(n+2)}{n}, & \text { if } n \mbox { is even.}
		\end{cases}
		$$
		Fix $\rho_{0}\in C_{0}^{\infty}((1,2))$ and $\rho_1 \in C_0^{\infty}\left(\mathbb{R}^n \times I\right)$, where $I$ is a compact interval. Let $a\in C^{\infty}\big(\mathbb{R}^{n}\setminus0\big)$ be homogeneous of degree $0$. For $x\in \mathbb R^n$, $t\in\mathbb R$, define
		\begin{equation*}
			F_{j}f(x,t)
			=
			\rho_{1}(x,t)\int_{\mathbb{R}^{n}}e^{i\langle x,\xi\rangle}e^{it|\xi|}\rho_{0}\big(|2^{-j}\xi|\big)a(\xi)\hat{f}(\xi)d\xi,\quad f\in \mathcal{S}(\mathbb{R}^n;L_p(\mathcal{M})).
		\end{equation*}
		Then for $j > 0$, $p\geq \bar{p}_{n}$, $\sigma<1/p$,
		\begin{equation}\label{conjecFIO}
			\left\| F_jf\right\|_{L^{p}(\mathcal{A})}\lesssim2^{u j}\|f\|_{L^p(\mathcal{N})}
		\end{equation}
		where $\mathcal{A}:=L_{\infty}(\mathbb{R}^n\times \mathbb{R})\overline{\otimes}\mathcal{M}$ and $u=s_{p}-\sigma$. 
		Especially, if $p=4$, one has $u>\frac{1}{4}(n-2)$.
	\end{assumption}

	The following is the main result of this paper.
	
	\begin{theorem}\label{Main}
		Suppose that (\ref{conjecFIO}) holds.
		\begin{enumerate}[\rm (i)]
			\item\label{main1} Let $n\geq 2$ and $2\leq p\leq \infty$. We have
			\begin{equation}\label{main}
				\left\| \sup_{t>0}{}^+ \mathscr{M}_t^{\alpha}f \right\|_{L_p(\mathcal{N})}
				\lesssim \|f\|_{L_p(\mathcal{N})}
			\end{equation}
			for all $\alpha\in \mathbb{C}$ 
			such that
			\begin{equation}\label{alpha}
				\operatorname{Re}\alpha
				>
				\max\left\lbrace -\frac{n-3}{p}-\frac{1}{2},\, \frac{n-3}{p}-\frac{n}{2}+1\right\rbrace.
			\end{equation}
			\item \label{1} Let $n\geq 2$ and $4\leq p<\infty$. For any $f\in L_p(\mathcal{N})$ and $\alpha\geq 0$, we have 
			$$
			\mathscr{M}^{\alpha}_tf \overset{a.u}{\longrightarrow} f \quad \text{  as  }\quad t\to 0.
			$$
			\item \label{2} Let $n\geq 2$ and $2\leq p<\infty$. Let $\alpha\in \mathbb C$ satisfy (\ref{alpha}). For any $f\in L_p(\mathcal{N})$,
			$$
			\mathscr{M}^{\alpha}_tf \overset{b.a.u}{\longrightarrow} f \quad \text{  as  }\quad t\to 0.
			$$
		\end{enumerate}
	\end{theorem}
	Assumption \ref{assum} can be considered as the noncommutative counterpart  of the local smoothing estimates (see for instance \cite[Conjecture 1.3]{GLMX2023}).  In classical setting, the local smoothing estimates have been extensively studied in numerous papers, see for instance \cite{BD2015, GS2010, GWZ2020, LW2002, MSS1992, MYZ2017, S2017, W2000} and the references therein. To make our result more convincible, we mention here that Hong, Lai and Wang \cite{HLW} recently obtained a noncommutative version of  local smoothing estimate which can be summarized as follows:
	$$\left\|F_j f(x, t)\right\|_{L_p(\mathcal{A})}\lesssim 2^{\gamma j}\|f\|_{L_p(\mathcal{N})}$$ with
	$$
	\gamma
	>
	\begin{cases}
		\frac{1}{2}\left(\frac{1}{2}-\frac{1}{p}\right),\,2<p \leq 4,\\
		\frac{1}{2}\left(1-\frac{3}{p}\right),\,4 \leq p<\infty.
	\end{cases}
	$$
	%
	This serves as a nice example of Assumption \ref{assum} and illustrates that our maximal inequality holds true in this setting.
	We will focus on improving the results of noncommutative version of local smoothing estimates by verifying Assumption \ref{assum} in-depth in our future work. This will be an interesting direction with lots of applications.
	
	The paper is organized as follows. In section \ref{sec2}, we review some preliminary notation and results on the noncommutative $L_p$-spaces, maximal inequalities and pointwise convergence. In section \ref{sec3}, we reduce the proof of our main result to that of Proposition \ref{lemma12}. 
	Section \ref{sec4} is devoted to the proof of Proposition \ref{lemma12}.
	
	Throughout this paper, the letter $C$ stands for a positive finite constant which is independent of the essential variables, not necessarily the same one in each occurrence. The notion $A \lesssim B$ means that $A \leq C B$ for some constant $C$, and $A \approx B$ means that $A \lesssim B$ and $B \lesssim A$. By the notation $C_{\varepsilon}$ we mean that the constant depends on the parameter $\varepsilon$.
	Let $\mathbb{Z}_{+}$ denote the set of all nonnegative integers.
	For a function $f$ on $\mathbb{R}^n$, define $\hat{f}$  and  $\check{f}$  the Fourier transform and the inversion Fourier transform of $f$ by
	$$
	\hat{f}(\xi)
	=
	\int_{\mathbb{R}^n} e^{-2 \pi i\langle x, \xi\rangle} f(x) d x, \quad
	\check{f}(\xi)
	=
	\int_{\mathbb{R}^d} e^{2 \pi i\langle x, \xi\rangle} f(x) d x.
	$$

	\section{Preliminaries}\label{sec2}
	In this section, we introduce some basic definitions and facts related to noncommutative $L_p$ spaces, maximal functions and pointwise convergence.
	
	Let $\mathcal{M}$ denote a semifinite von Neumann algebra equipped with a normal semifinite faithful (abbrieviated as n.s.f.) trace $\tau$. Let $\mathcal{M}_{+}$ be the positive part of $\mathcal{M}$ and $\mathcal{S}_{\mathcal{M}+}$ be the set of all $x \in \mathcal{M}_{+}$ such that $\tau(\operatorname{supp}(x))<\infty$, where $\operatorname{supp}(x)$ denotes the support of $x$. Let $\mathcal{S}_{\mathcal{M}}$ be the linear span of $\mathcal{S}_{\mathcal{M}_+}$. Then $\mathcal{S}_{\mathcal{M}}$ is a $w^*$-dense $*$-subalgebra of $\mathcal{M}$. Given $1 \leq p<\infty$ and $x \in \mathcal{S}_{\mathcal{M}}$, we define
	$$
	\|x\|_p=\left(\tau\left(|x|^p\right)\right)^{1 / p},
	$$
	where $|x|=\left(x^* x\right)^{1 / 2}$ is the modulus of $x$. Then $\|\cdot\|_p$ is a norm on $\mathcal{S}_{\mathcal{M}}$ and $\left(\mathcal{S}_{\mathcal{M}},\|\cdot\|_p\right)$ forms a normed space. The completion of $\left(\mathcal{S}_{\mathcal{M}},\|\cdot\|_p\right)$ is the so-called noncommutative $L_p$ space associated with $(\mathcal{M}, \tau)$, denoted by $L_p(\mathcal{M})$. As usual, we set $L_{\infty}(\mathcal{M})=\mathcal{M}$ equipped with the operator norm $\|\cdot\|_{\mathcal{M}}$ and denote by $L_p^+(\mathcal{M})$ the positive part of $L_p(\mathcal{M})$. Let $L_0(\mathcal{M})$ denote the space of all closed densely defined operators on $H$ measurable with respect to $(\mathcal{M}, \tau)$, where $H$ is the Hilbert space on which $\mathcal{M}$ acts. Then the elements of $L_p(\mathcal{M})$ can be viewed as closed densely defined operators on $H$. For $a, b \in \mathcal{M}$ and $0<\delta<1$, the following monotone properties are valid:
	\begin{enumerate}
		\item $\|a\|_{L_p(\mathcal{M})} \leq\|b\|_{L_p(\mathcal{M})}$, if $0\leq a \leq b$;
		\item $a^\delta \leq b^\delta$, if $0 \leq a \leq b$ and $0<\delta<1$.
	\end{enumerate}
	We refer the reader to \cite{FK1986, PX2003} for more basic properties such as Minkowski's inequality, H\"{o}lder's inequality, dual property, real and complex interpolation  on noncommutative $L_p$-spaces.
	
	Throughout the paper, we denote by $\mathcal N$  the  von Neumann algebra $L_{\infty}(\mathbb{R}^n) \overline{\otimes} \mathcal{M}$  equipped with the tensor trace $\int_{\mathbb{R}^n} d x \otimes \tau$. Since $L_2(\mathcal{M})$ is a Hilbert space, it follows from the vector-valued Plancherel theorem that
	\begin{equation}\label{plan}
		\|\widehat{f}\|_{L_2(\mathcal{N})}=\|\widehat{f}\|_{L_2(\mathbb R^n, L_2(\mathcal M))}=\|f\|_{L_2(\mathbb R^n, L_2(\mathcal M))}=\|f\|_{L_2(\mathcal{N})},
	\end{equation}
	This fact will be used in the sequel.
	
	We now turn to introduce noncommutative vector-valued spaces.
	Following Pisier \cite{P1998} and Junge \cite{J2002}, we give the following definition.
	\begin{definition}
		Let $I$ be an index set. Given $1 \leq p \leq \infty$, we define $L_p(\mathcal{M} ; \ell_{\infty}(I))$ as the space of all families $(x_i)_{i \in I}$ in $L_p(\mathcal{M})$ which can be factorized as $x_i=a y_i b$ with $a, b \in L_{2 p}(\mathcal{M})$ and $(y_i)_{i \in I} \subset L_{\infty}(\mathcal{M})$. The norm of $(x_i)_{i\in I}$ in $L_p(\mathcal{M} ; \ell_{\infty}(I))$ is defined as
		$$
		\left\|\left(x_i\right)_{i \in I}\right\|_{L_p\left(\mathcal{M} ; \ell_{\infty}(I)\right)}=\inf \left\{\|a\|_{L_{2p}(\mathcal{M})} \sup_{i \in I}\left\|y_i\right\|_{L_{\infty}(\mathcal{M})}\|b\|_{L_{2 p}(\mathcal{M})}\right\},
		$$
		where the infimum is taken over all factorizations as above.
	\end{definition}
	
	For simplicity, we denote $\|(x_i)_{i\in I}\|_{L_p(\mathcal{M} ; \ell_{\infty}(I))}$ by  $\|\sup _{i\in I}^{+} x_i\|_{L_p(\mathcal{M})}$ in the following. Note that  $\sup _{i\in I}^{+} x_i$ means nothing but just a notation since $\sup _{i\in I} x_i$ does not make any sense in the noncommutative setting.

	Let $x=(x_i)_{i \in I}$ be a sequence of positive operators in $L_p(\mathcal{M})$. Then $x=(x_i)_{i \in I}$ belongs to $L_p\left(\mathcal{M} ; \ell_{\infty}(I)\right)$ if and only if there is $a \in L_p^+(\mathcal{M})$ such that $0 < x_i \leq a$ for all $i \in I$. Moreover, in this case,
	$$
	\|x\|_{L_p\left(\mathcal{M} ; \ell_{\infty}(I)\right)}
	=
	\inf \big\{\|a\|_{L_p(\mathcal{M})}: a \in L_p^+(\mathcal{M}) \text{ such that } 0 < x_i \leq a, \forall i \in I\big\} .
	$$
	Similarly if $(x_i)_{i \in I}$ is a sequence of self-adjoint operators in $L_p(\mathcal{M})$. Then $x=(x_i)_{i \in I}$ belongs to $L_p\left(\mathcal{M} ; \ell_{\infty}(I)\right)$ if and only if there is $a \in L_p^+(\mathcal{M})$ such that $-a \leq x_i \leq a$ for all $i\in I$, and
	$$
	\|x\|_{L_p\left(\mathcal{M} ; \ell_{\infty}(I)\right)}
	=
	\inf \left\{\|a\|_p: a \in L_p^{+}(\mathcal{M}) \text{ such that }-a \leq x_i \leq a, \forall i \in I\right\} .
	$$
	These properties can be found in \cite[Remark 4.1]{CXY2013}.
	We now recall a lemma which will be used in the proof of our main result (see e.g. \cite[Theorem 4.3]{CXY2013} or \cite[Lemma 5.1]{HLX2022}) .
	
	\begin{lemma}\label{lem_psi}
		Let $\psi$ be a non-negative radial function on $\mathbb{R}^n$ such that $\psi(x) \lesssim \frac{1}{(1+|x|)^{n+\delta}}$ for $x \in \mathbb{R}^n$ with some $\delta>0$. Let $\psi_{t}(x)=\frac{1}{t^n} \psi\left(\frac{x}{t}\right)$ for $x \in \mathbb{R}^n$ and $t>0$. Let $1<p \leq \infty$. Then for $f \in L_p(\mathcal{N})$,
		$$
		\left\|\sup _{t>0}{}^{+} \psi_{t} * f\right\|_{L_p(\mathcal{N})}
		\lesssim
		\left\|\sup _{t>0}{}^{+} M_{t} f\right\|_{L_p(\mathcal{N})}
		\lesssim
		\left\| f\right\|_{L_p(\mathcal{N})}
		,
		$$
		where $M_t$ is the Hardy-Littlewood averaging operator
		$$
		M_t f(x)
		=
		\frac{1}{t^n}\int_{|x-y|\leq t}f(y)dy.
		$$
	\end{lemma}
	
	We will also need the following lemma. The verification is routine, however we provide the proof here for convenience.
	
	\begin{lemma}\label{norm}
		Let $1<p<\infty$ and $x=(x_t)_{t>0}$ be a sequence of positive operators in $L_p(\mathcal{N};\ell_{\infty})$. Then
		$$
		\left\|\sup_{t>0}{}^+ x_t\right\|_{L_p(\mathcal{N})}^p
		\leq
		\sum_{k \in \mathbb{Z}}\left\|\sup_{t\in I_k}{}^+ x_t\right\|_{L_p(\mathcal{N})}^p, \quad I_k:=[2^{-k},2^{-k+1}],\quad k\in\mathbb{Z}.
		$$
	\end{lemma}
	
	\begin{proof}
		Given $\varepsilon>0$, $k\leq 0$, there exists a positive operator $a_k$ such that for all $t\in I_k$, $x_t\leq a_k$ and
		$$
		\|a_k\|_{L_p(\mathcal{N})}^p\leq \left\|\sup_{t\in I_k}{}^+ x_t\right\|_{L_p(\mathcal{N})}^p+2^{k-2}\varepsilon.
		$$
		The same estimate still holds for $k>0$, except that the last term $2^{k-2}\varepsilon$ is replaced by $2^{-k-2}\varepsilon$.
		Set $a=\left(\sum_{k\in\mathbb{Z}}|a_k|^p \right)^{1/p}$. We clearly have
		$x_t\leq a$ for all $t>0$.
		Indeed, for every $t>0$, there exists $k\in \mathbb{Z}$ such that $t\in I_k$. Therefore,
		$$
		x_k\leq a_k\leq \left(\sum_{k\in\mathbb{Z}}|a_k|^p \right)^{1/p}=a.
		$$
		Moreover,
		$$
		\begin{aligned}
			\|a\|_{L_p(\mathcal{N})}^p
			=&
			\sum_{k\in \mathbb{Z}} \|a_k\|_{L_p(\mathcal{N})}^p
			=
			\sum_{k\leq 0}\|a_k\|_{L_p(\mathcal{N})}^p+\sum_{k>0}\|a_k\|_{L_p(\mathcal{N})}^p\\
			\leq&
			\sum_{k\leq 0} \left(\left\|\sup_{t\in I_k}{}^+ x_t \right\|_{L_p(\mathcal{N})}^p +2^{k-2}\varepsilon\right)
			+
			\sum_{k>0} \left(\left\|\sup_{t\in I_k}{}^+ x_t \right\|_{L_p(\mathcal{N})}^p +2^{-k-2}\varepsilon\right)\\
			\leq&
			\sum_{k\in\mathbb{Z}}\left\|\sup_{t\in I_k}{}^+ x_t\right\|_{L_p(\mathcal{N})}^p+\varepsilon.
		\end{aligned}
		$$
		Hence, we conclude that
		$$
		\left\|\sup_{t>0}{}^+ x_t\right\|_{L_p(\mathcal{N})}^p
		\leq
		\sum_{k\in \mathbb{Z}}\left\|\sup_{t\in I_k}{}^+ x_t\right\|_{L_p(\mathcal{N})}^p,
		$$
		which is desired estimate.
	\end{proof}

	The noncommutative Marcinkiewicz interpolation theorem in $L_p(\mathcal{M},\ell_{\infty})$ will play an important role in the study of our maximal inequalities. Recall that this powerful interpolation result was established by Junge and Xu in \cite[Theorem 3.1]{JX2007}. To describe their result, let us introduce the following definition.

	\begin{definition}
		Consider a family of maps $\Phi_t: L_p(\mathcal{M}) \rightarrow L_0(\mathcal{M})$ for $t>0$.
		\begin{enumerate}[\rm (i)]
			\item Let $1\leq p\leq \infty$. We say that $(\Phi_t)_{t>0}$ is of strong-type $(p, p)$ with constant $C$ if
			$$
			\left\|\sup_{t>0}{}^+ \Phi_t(x)\right\|_p \leq C\|x\|_p, \quad x \in L_p(\mathcal{M}) .
			$$
			\item Let $1\leq p<\infty$. We say that $(\Phi_t)_{t>0}$ is of weak-type $(p, p)$ with constant $C$ if for any $x \in L_p(\mathcal{M})$ and any $\lambda>0$ there is a projection $e \in \mathcal{M}$ such that
			$$
			\left\|e \Phi_t(x) e\right\|_{\infty} \leq \lambda, \quad \forall\,t>0 \quad \text { and } \quad \tau\left(e^{\perp}\right) \leq\left[C \frac{\|x\|_p}{\lambda}\right]^p .
			$$
		\end{enumerate}
	\end{definition}
	
	\begin{remark}
		It is not hard to see that for any $1\leq p<\infty$, strong-type $(p, p)$ implies weak-type $(p, p)$.
	\end{remark}
	
	Now, the noncommutative Marcinkiewicz interpolation theorem due to Junge and Xu \cite[Theorem 3.1]{JX2007} can be formulated as follows.
	
	\begin{lemma}\label{interpolation}\cite[Theorem 3.1]{JX2007}
		Let $1 \leq p_0<p_1 \leq \infty$. Let $S=(S_t)_{t >0}$ be a family of maps from $L_{p_0}^{+}(\mathcal{M})+$ $L_{p_1}^{+}(\mathcal{M})$ into $L_0^{+}(\mathcal{M})$. Assume that $S$ is subadditive in the sense that $S_t(x+y) \leq S_t(x)+S_t(y)$ for all $t>0$. If $S$ is of weak type $(p_0, p_0)$ with constant $C_0$ and of type $(p_1, p_1)$ with constant $C_1$, then $S$ is of type $(p, p)$ for any $p_0<p<p_1$.
		Moreover, if $p_1 \geq 2 p_0$ (a condition that is almost always satisfied in application), the type $(p, p)$ constant $C_p$ of $S$ is controlled by
		$$
		C_p \leq C C_0^{1-\theta} C_1^\theta\left(\frac{1}{p_0}-\frac{1}{p}\right)^{-2}\left(\frac{1}{p}-\frac{1}{p_1}\right)^{-1},
		$$
		where $\theta$ is determined by $1 / p=(1-\theta) / p_0+\theta / p_1$ and $C$ is a universal constant.
	\end{lemma}

	We conclude the section with the notion of almost uniform convergence introduced by Lance \cite{L1976}.
	
	\begin{definition}
		Let $x_k, x \in L_0(\mathcal{M})$.
		\begin{enumerate}[\rm (i)]
			\item $\left(x_k\right)$ is said to converge almost uniformly (a.u. in short) to $x$ if for any $\delta>0$, there exists a projection $e \in \mathcal{M}$ such that
			$$
			\tau\left(e^{\perp}\right)<\delta \quad \text { and } \quad \lim _{k \rightarrow \infty}\left\|\left(x_k-x\right) e\right\|_{L_\infty(\mathcal{M})}=0 .
			$$
			\item $\left(x_k\right)$ is said to converge bilaterally almost uniformly (b.a.u. in short) to $x$ if for any $\delta>0$, there exists a projection $e \in \mathcal{M}$ such that
			$$
			\tau\left(e^{\perp}\right)<\delta \quad \text { and } \quad \lim _{k \rightarrow \infty}\left\|e\left(x_k-x\right) e\right\|_{L_\infty(\mathcal{M})}=0 .
			$$
		\end{enumerate}
	\end{definition}
	Obviously, $x_k \stackrel{\text { a.u }}{\longrightarrow} x$ implies $x_k \stackrel{\text { b.a.u }}{\longrightarrow} x$. Both convergences defined above are equivalent to the usual almost everywhere convergence due to the Egorov theorem in the commutative probability space.

	\section{Proof of the main result}\label{sec3}
	In this section we reduce Theorem \ref{Main} \eqref{main1} to  Proposition \ref{lemma12} and the proof of Proposition \ref{lemma12} will be postponed until the next section. Then, applying Theorem \ref{Main} \eqref{main1}, we complete the proof of \eqref{1} and \eqref{2}.
	
	
	Let us first concentrate on the maximal inequality \eqref{main}; namely Theorem \ref{Main} \eqref{main1}. Assume that $f$ is positive since it can be decomposed into $f_1-f_2+i(f_3-f_4)$ with positive $f_k$ $(k=1,2,3,4)$.
	Taking a function
	$
	\varphi \in C_0^{\infty}(\mathbb{R})
	$
	such that
	$
	\operatorname{supp} \varphi \subset [1 / 2,2]
	$
	to form a partition of unity
	$$
	\sum_{j\in \mathbb{Z}} \varphi\left(2^{-j} s\right)=1, \quad s>0,
	$$
	we then denote
	$$
	\begin{gathered}
		\varphi_j(\xi)=\varphi\left(2^{-j}|\xi|\right), \quad
		\varphi_0(\xi)=1-\sum_{j=1}^{\infty} \varphi_j(\xi)
	\end{gathered}
	$$
	for $\xi \in \mathbb{R}^n$ and $j\in \mathbb{Z}_+$.
	Define the operator $\mathscr{M}_{j, t}^\alpha$ by
	\begin{equation}\label{malphajt}
		\mathscr{M}_{j, t}^\alpha f(x)
		=
		(\varphi_j(t\cdot)\widehat{m}_\alpha(t \cdot) \widehat{f})^{\vee}(x).
	\end{equation}
	It is obvious that
	\begin{equation}\label{malpha}
		\mathscr{M}_t^\alpha f(x)=\sum_{j=0}^{\infty} \mathscr{M}_{j, t}^\alpha f(x).
	\end{equation}
	Then the conclusion (\ref{main}) can be immediately obtained from (\ref{malpha}) and the following result.

	\begin{theorem}\label{mainprop}
		For $\alpha$ and $p$ satisfying (\ref{alpha}), there exists some $\mu<0$ such that
		\begin{equation}\label{2.16}
			\left\|\sup _{t>0}{}^{+} \mathscr{M}_{j, t}^\alpha f\right\|_{L_p(\mathcal{N})}
			\lesssim
			2^{ \mu j}\|f\|_{L_p(\mathcal{N})},\quad j\geq 0.
		\end{equation}
	\end{theorem}
	

	\begin{proof}We divide the proof into the cases $j=0$ and $j> 0$.
		Let us first treat the special case $j=0$.
		In view of (\ref{malphajt}) we write $\mathscr{M}_{0, t}^\alpha f(x)$ as $\Phi_{0,t} * f(x)$, 
		where
		$$
		\Phi_{0,t}(x)
		=
		\int_{\mathbb{R}^n}\widehat{m}_\alpha(t \xi) \varphi_0(t\xi) e^{2 \pi i x \cdot \xi}d\xi.
		$$
		Given $\xi=(\xi_1,\cdots,\xi_n)\in \mathbb{R}^n$, denote by $\Delta_{\xi}$ the Laplacian operator in $\xi$. 
		For $t>0$ 
		and arbitrary $N>0$, 
		via integration by parts, we obtain
		\begin{equation}\label{F1}
			\begin{aligned}
				\left|\Phi_{0,t}(x)\right|
				& =
				\frac{1}{t^n}\left|\int_{\mathbb{R}^n}\widehat{m}_\alpha(\xi) \varphi_0\left(\xi\right)\left(1+\left|\frac{x}{t}\right|^2\right)^{-N}\left(1-\frac{\Delta_{\xi}}{4 \pi^2}\right)^N e^{2 \pi i \xi \cdot \frac{x}{t}}d\xi\right| \\
				& =
				\frac{1}{t^n}\left(1+\left|\frac{x}{t}\right|^2\right)^{-N}
				\left|\int_{\mathbb{R}^n}e^{2 \pi i \xi \cdot \frac{x}{t}}\left(1-\frac{\Delta_{\xi}}{4 \pi^2}\right)^N\left[\widehat{m}_\alpha(\xi) \varphi_0\left(\xi\right)\right]d\xi\right| \\
				&\lesssim
				\frac{1}{t^{n}}\left(1+\left|\frac{x}{t}\right|\right)^{-N},
			\end{aligned}
		\end{equation}
		where the last inequality is due to the fact that $\widehat{m}_\alpha(\xi ) \varphi_0(\xi)$ is smooth and supported in $|\xi| \leq 2$.
		Hence,
		$$
		\begin{aligned}
			\mathscr{M}_{0,t}^\alpha f(x)
			=
			\Phi_{0,t}*f(x)
			\lesssim
			\frac{1}{t^n}\left(1+\left|\frac{\cdot}{t}\right|\right)^{-N}*f(x).
		\end{aligned}
		$$
		Then, by Lemma \ref{lem_psi}, 
		we conclude that
		$$
		\begin{aligned}
			\left\|\sup_{t>0}{}^+\mathscr{M}_{0,t}^\alpha f\right\|_{L_p(\mathcal{N})}
			\lesssim&
			\left\|\sup _{t>0}{}^+ \frac{1}{t^n}\left(1+\left|\frac{\cdot}{t}\right|\right)^{-N}*f\right\|_{L_p(\mathcal{N})} \\
			\lesssim&
			\left\|\sup _{t>0}{}^+ M_tf\right\|_{L_p(\mathcal{N})}
			\lesssim
			\left\|f\right\|_{L_p(\mathcal{N})}.
		\end{aligned}
		$$
		
		Next we consider the case $j>0$. It is sufficient to prove
		\begin{equation}\label{2.17}
			\left\|\sup_{2^{-k}\leq t\leq 2^{-k+1}}{}^+ \mathscr{M}_{j, t}^\alpha f\right\|_{L_p(\mathcal{N})}
			\lesssim
			2^{\mu j}\|f\|_{L_p(\mathcal{N})}, \quad \forall k \in \mathbb{Z}.
		\end{equation}
		Indeed, it follows from Lemma \ref{norm} that
		$$
		\left\|\sup_{t>0}{}^+ \mathscr{M}_{j, t}^\alpha f\right\|_{L_p(\mathcal{N})}
		\leq
		\left(\sum_{k\in\mathbb{Z}}\left\|\sup_{2^{-k}\leq t\leq 2^{-k+1}}{}^+ \mathscr{M}_{j, t}^\alpha f\right\|_{L_p(\mathcal{N})}^p\right)^{1/p}.
		$$
		Denote by $\widehat{\dot{\Delta}_{\ell} f}(\xi)=\varphi\left(2^{-\ell}|\xi|\right) \widehat{f}(\xi)$ for all $\ell \in \mathbb{Z}$ and notice that
		$$
		\mathscr{M}_{j, t}^\alpha f(x)=\sum_{|\ell| \leq 100} \mathscr{M}_{j, t}^\alpha\left(\dot{\Delta}_{j+k+\ell} f\right)(x),
		$$
		whenever $2^{-k} \leq t \leq 2^{-k+1}$ for $k \in \mathbb{Z}$. We have
		$$
		\left\|\sup_{t>0}{}^+ \mathscr{M}_{j, t}^\alpha f\right\|_{L_p(\mathcal{N})}
		\leq
		\left(\sum_{k\in\mathbb{Z}}\left\|\sup_{2^{-k}\leq t\leq 2^{-k+1}}{}^+ \mathscr{M}_{j, t}^\alpha \left(\sum_{|\ell| \leq 100} \dot{\Delta}_{j+k+\ell} f\right)\right\|_{L_p(\mathcal{N})}^p\right)^{1/p}.
		$$
		Using \eqref{2.17} and the noncommutative Littlewood-Paley inequality \cite[Theorem B2]{MP2009}, we obtain
		$$
		\begin{aligned}
			\left\|\sup_{t>0}{}^+ \mathscr{M}_{j, t}^\alpha f\right\|_{L_p(\mathcal{N})}
			\lesssim&
			2^{\mu j}\left(\sum_{k\in\mathbb{Z}}\left\|\sum_{|\ell| \leq 100} \dot{\Delta}_{j+k+\ell} f\right\|_{L_p(\mathcal{N})}^p\right)^{1/p}\\
			\lesssim&
			2^{\mu j}\left\|\left(\sum_{k \in \mathbb{Z}}\left|\dot{\Delta}_k f(x)\right|^2\right)^{\frac{1}{2}}\right\|_{L_p(\mathcal{N})}
			\lesssim
			2^{\mu j}\|f\|_{L_p(\mathcal{N})},
		\end{aligned}
		$$
		which is the desired estimate.
		
		Consequently, it remains to verify (\ref{2.17}). It is straightforward that
		$$
		\left\|\sup _{2^{-k}\leq t\leq 2^{-k+1}}{}^+\mathscr{M}_{j, t}^\alpha f\right\|_{L_p(\mathcal{N})}
		=
		\left\|\sup _{1\leq t\leq 2}{}^+\mathscr{M}_{j, t}^\alpha\left(f_{k}\right)\left(2^k \cdot\right)\right\|_{L_p(\mathcal{N})},
		$$
		where $f_{k}(x)=f\left(2^{-k} x\right)$. Hence (\ref{2.17}) is a consequence of the following proposition.
	\end{proof}
	
	\begin{proposition}\label{lemma12}
		With all notation above, there exists some $\mu<0$ such that for $2\leq p\leq \infty$,
		$$
		\left\|\sup_{1\leq t\leq2}{}^+ \mathscr{M}_{j,t}^\alpha f\right\|_{L_{p}(\mathcal{N})}
		\lesssim
		2^{\mu j}\|f\|_{L_{p}(\mathcal{N})}, \quad j> 0.
		$$
	\end{proposition}

	The proof of  Proposition \ref{lemma12} will be given in Section \ref{sec4}.
	Now we continue dealing with the pointwise convergence stated in Theorem \ref{Main} (\ref{1}) and (\ref{2}). Before that we give a useful lemma.
	\begin{lemma}\label{lammtf2}
		With all notation above, one has for $\alpha\geq 0$ that
		$$
		\left|\mathscr{M}_t^{\alpha}f(x)\right|^2
		\lesssim
		\mathscr{M}_t^{\alpha}(|f|^2)(x).
		$$
	\end{lemma}
	\begin{proof}
		By H\"{o}lder's inequality (see e.g. \cite[(1.13)]{M2007}), we have
		\begin{align*}
			\left|\mathscr{M}_t^{\alpha}f(x)\right|^2=|m_{\alpha,t}*f(x)|^2&=\left|\int_{\mathbb{R}^n} m_{\alpha,t}^{\frac12}(y)(m_{\alpha,t}^{\frac12}(y)f(x-y))dy\right|^2\\
			&\leq \int_{\mathbb{R}^n} m_{\alpha, t}(y) dy\int_{\mathbb{R}^n} m_{\alpha,t}(y)|f|^2(x-y))dy\\
			&=\int_{\mathbb{R}^n} m_{\alpha, t}(y) dy\cdot\mathscr{M}_t^{\alpha}(|f|^2)(x).
		\end{align*}
		Therefore, it suffices to show that $\int_{\mathbb{R}^n} m_{\alpha, t}(y) dy<\infty$. For $\alpha>0$, we have
		$$
		\begin{aligned}
			\int_{\mathbb{R}^n} m_{\alpha, t}(y) dy
			&=
			\frac{1}{\Gamma(\alpha)} \int_{|y|\leq1}\left(1-|y|^2\right)^{\alpha-1} d y \\
			& =
			\frac{1}{\Gamma(\alpha)} \int_{\mathbb{S}^{n-1}} \int_0^1\left(1-r^2\right)^{\alpha-1} r^{n-1} d r d \sigma(\theta)
			\leq
			\frac{|\mathbb{S}^{n-1}|}{\alpha \Gamma(\alpha)}.
		\end{aligned}
		$$
		On the other hand, for $\alpha=0$,
		$$
		\int_{\mathbb{R}^n} m_{0, t}(y) dy
		=
		\lim _{\alpha \rightarrow 0^{+}} \frac{1}{\Gamma(\alpha)} \int_{|y|\leq1}\left(1-|y|^2\right)^{\alpha-1} d y
		=\int_{\mathbb{S}^{n-1}} d \sigma(\theta)=|\mathbb{S}^{n-1}|,
		$$
		where we have used the fact
		$$
		\lim _{\alpha \rightarrow 0^{+}} \frac{|\cdot|^{\alpha-1}}{\Gamma(\alpha)}=\delta
		$$
		(here, $\delta$ denotes the Dirac distribution at zero). The proof is complete.
	\end{proof}
	
	We are now ready to prove Theorem \ref{Main} \eqref{1}.
	Suppose that $2\leq p<\infty$. Fix $\delta >0$. By the density of $C_0^{\infty}(\mathbb{R}^n)\otimes S_{\mathcal{M}}$ in $L_{2p}(\mathcal{N})$, for every $l\geq 1$, there exists
	$
	g_l=\sum_{i=1}^l\phi_i\otimes m_i\in C_0^{\infty}(\mathbb{R}^n)\otimes S_{\mathcal{M}}
	$
	such that $\|f-g_l\|_{L_{2p}(\mathcal{N})}^{2p}\leq \frac{\delta}{2^ll^{2p}}$.
	Now applying Lemma \ref{lammtf2} and the strong-type $(p, p)$ estimate established in Theorem \ref{Main} (\ref{main1}) to each $|f-g_l|^2$, 
	there exists a projection $e_{1,l}\in \mathcal{N}$ such that
	$$
	\begin{aligned}
		\sup_{t>0}\|e_{1,l} \mathscr{M}^\alpha_t(f-g_l)\|_{L_\infty(\mathcal{N})}
		=&
		\sup_{t>0}\|e_{1,l} |\mathscr{M}^\alpha_t(f-g_l)|^2e_{1,l}\|_{L_\infty(\mathcal{N})}^{1/2}\\
		\lesssim&
		\sup_{t>0}\|e_{1,l} \mathscr{M}^\alpha_t(|f-g_l|^2)e_{1,l}\|_{L_\infty(\mathcal{N})}^{1/2}
		\leq
		\frac{1}{l},
	\end{aligned}
	$$
	and
	$$
	\int_{\mathbb{R}^n}\otimes \tau(e_{1,l}^{\perp})
	<
	l^{2p}\||f-g_l|^2\|_{L_p(\mathcal{N})}^{p}
	=
	l^{2p}\|f-g_l\|_{L_{2p}(\mathcal{N})}^{2p}
	\leq
	\frac{\delta}{2^l}.
	$$
	Let $e_1=\wedge_le_{1,l}$. Then we have
	$$\int_{\mathbb{R}^n}\otimes \tau(e_1^{\perp})<\delta
	\quad \mbox{and}\quad \sup_{t>0}\|e_1 \mathscr{M}^\alpha_t(f-g_l)\|_{L_\infty(\mathcal{N})}\leq\frac{1}{l},\quad l\geq 1.
	$$
	The same argument ensures that there exists  $e_{2} \in \mathcal{N}$ such that
	$$\int_{\mathbb{R}^n}\otimes \tau(e_2^{\perp})<\delta
	\quad\text{and}\quad
	\|e_2(f-g_l)\|_{L_\infty(\mathcal{N})}\leq\frac{1}{l},\quad l\geq 1.
	$$
	
	Recall that for $\phi_i\in C_0^{\infty}(\mathbb{R}^n)$. We have
	$$
	\lim_{t\to 0^+}\left|\int_{\mathbb{R}^n} \widehat{\phi_i}(\xi)(\widehat{m_{\alpha}}(t\xi)-1) e^{2\pi ix\cdot\xi}d\xi\right|
	=
	0
	$$
	due to the asymptotic behavior of Bessel function in (\ref{1.2}). This together with
	$$
	\begin{aligned}
		\mathscr{M}^\alpha_tg_l(x)-g_l(x)
		=&
		\sum_{i=1}^l \left[\int_{\mathbb{R}^n} \widehat{\phi_i}(\xi)(\widehat{m_{\alpha}}(t\xi)-1) e^{2\pi ix\cdot\xi}d\xi\right]\otimes m_i 
	\end{aligned}
	$$
	implies that
	$
	\lim_{t\to 0^+}\|\mathscr{M}_t^\alpha g_l-g_l\|_{L_\infty(\mathcal{N})}=0.
	$
	Therefore, for any $l\geq 1$,
	there exists a positive constant $N_l$ such that for any $1/t>N_l$
	$$
	\|\mathscr{M}_t^\alpha g_l-g_l\|_{L_\infty(\mathcal{N})}\leq\frac{1}{l}.
	$$
	
	Let $e=e_1\wedge e_2$. Putting the above estimates together, we obtain
	\begin{align*}
		&\|e(\mathscr{M}^\alpha_tf-f)\|_{L_\infty(\mathcal{N})}\\
		&\leq
		\|e(\mathscr{M}^\alpha_tg_l-g_l)\|_{L_\infty(\mathcal{N})}
		+
		\|e\mathscr{M}^\alpha_t(f-g_l)\|_{L_\infty(\mathcal{N})}
		+
		\|e(f-g_l)\|_{L_\infty(\mathcal{N})}\\
		&\leq
		\|\mathscr{M}^\alpha_tg_l-g_l\|_{L_\infty(\mathcal{N})}
		+
		\|e_1\mathscr{M}^\alpha_t(f-g_l)\|_{L_\infty(\mathcal{N})}
		+
		\|e_2(f-g_l)\|_{L_\infty(\mathcal{N})}\leq \frac{3}{l},
	\end{align*}
	which yields
	$
	\lim_{t\to 0^+} \|e(\mathscr{M}^\alpha_tf-f)\|_{L_\infty(\mathcal{N})}=0
	$
	and finishes the proof of Theorem \ref{Main} (\ref{1}).
	
	Based on the above arguments, Theorem \ref{Main} (\ref{2}) can be proved with minor changes. We omit the details.

	\section{Proof of Proposition \ref{lemma12}}\label{sec4}
	
	In this section, we show that Proposition \ref{lemma12} is a natural conclusion of the combination of the noncommutative interpolation theorem (i.e., Lemma \ref{interpolation}) and the strong-type $(p,p)$ estimates of $\{\mathscr{M}_{j,t}^\alpha\}_{1\leq t\leq 2}$ with $p=2,4,\infty$.
	Before doing that, we state and prove a lemma needed in our proof.
	\begin{lemma}
		Let $m$ be a positive integer. Then we have
		\begin{equation}\label{NL}
			\left|\mathscr{M}_{j, t}^\alpha f\right|^{2m}
			-
			\left|\mathscr{M}_{j, 1}^\alpha f\right|^{2m}
			=
			\int_1^t\frac{\partial}{\partial s}|\mathscr{M}_{j, s}^\alpha f|^{2m}ds, \quad 1\leq t\leq2.
		\end{equation}
	\end{lemma}
	\begin{proof}
		It is easy to verify that
		$$
		\begin{aligned}
			& \left|\mathscr{M}_{j, t}^\alpha f\right|^{2m}
			-
			\left|\mathscr{M}_{j, 1}^\alpha f\right|^{2m}\\
			&\quad=
			[(\mathscr{M}_{j, t}^\alpha f)^*\mathscr{M}_{j, t}^\alpha f]^m-[(\mathscr{M}_{j, 1}^\alpha f)^*\mathscr{M}_{j, 1}^\alpha f]^m\\
			&\quad=
			[(\mathscr{M}_{j, t}^\alpha f)^*-(\mathscr{M}_{j, 1}^\alpha f)^*]\mathscr{M}_{j, t}^\alpha f|\mathscr{M}_{j, t}^\alpha f|^{2m-2}\\
			&\quad\quad+
			(\mathscr{M}_{j, 1}^\alpha f)^*\left(\mathscr{M}_{j, t}^\alpha f|\mathscr{M}_{j, t}^\alpha f|^{2m-2}-\mathscr{M}_{j, 1}^\alpha f|\mathscr{M}_{j, 1}^\alpha f|^{2m-2}\right)\\
			&\quad=
			\int_1^t\left[\frac{\partial}{\partial s}(\mathscr{M}_{j, s}^\alpha f)^*\right]\mathscr{M}_{j, t}^\alpha f|\mathscr{M}_{j, t}^\alpha f|^{2m-2}ds\\
			&\quad\quad+
			(\mathscr{M}_{j, 1}^\alpha f)^*\left(\mathscr{M}_{j, t}^\alpha f|\mathscr{M}_{j, t}^\alpha f|^{2m-2}-\mathscr{M}_{j, 1}^\alpha f|\mathscr{M}_{j, 1}^\alpha f|^{2m-2}\right).
		\end{aligned}
		$$
		We now rewrite
		$
		\mathscr{M}_{j, t}^\alpha f|\mathscr{M}_{j, t}^\alpha f|^{2m-2}
		$
		as
		$$
		\int_{s}^t\frac{\partial}{\partial s'}\left(\mathscr{M}_{j, s'}^\alpha f|\mathscr{M}_{j, s'}^\alpha f|^{2m-2}\right)ds'
		+
		\mathscr{M}_{j, s}^\alpha f|\mathscr{M}_{j, s}^\alpha f|^{2m-2},
		$$
		and interchange the order of double integral. This leads to the identity
		$$
		\begin{aligned}
			&\int_1^t\left[\frac{\partial}{\partial s}(\mathscr{M}_{j, s}^\alpha f)^*\right]\mathscr{M}_{j, t}^\alpha f|\mathscr{M}_{j, t}^\alpha f|^{2m-2}ds\\
			&\quad=
			\int_1^t\int_{1}^{s'}\left[\frac{\partial}{\partial s}(\mathscr{M}_{j, s}^\alpha f)^*\right]ds\frac{\partial}{\partial s'}\mathscr{M}_{j, s'}^\alpha f|\mathscr{M}_{j, s'}^\alpha f|^{2m-2}ds'\\
			&\quad \quad+
			\int_1^t\left[\frac{\partial}{\partial s}(\mathscr{M}_{j, s}^\alpha f)^*\right]\mathscr{M}_{j, s}^\alpha f|\mathscr{M}_{j, s}^\alpha f|^{2m-2}ds\\
			&\quad=
			\int_1^t(\mathscr{M}_{j, s'}^\alpha f)^*\frac{\partial}{\partial s'}\left[\mathscr{M}_{j, s'}^\alpha f|\mathscr{M}_{j, s'}^\alpha f|^{2m-2}\right]ds'\\
			&\quad \quad-
			\int_1^t(\mathscr{M}_{j, 1}^\alpha f)^*\frac{\partial}{\partial s'}\left[\mathscr{M}_{j, s'}^\alpha f|\mathscr{M}_{j, s'}^\alpha f|^{2m-2}\right]ds'\\
			&\quad \quad+
			\int_1^t\left[\frac{\partial}{\partial s}(\mathscr{M}_{j, s}^\alpha f)^*\right]\mathscr{M}_{j, s}^\alpha f|\mathscr{M}_{j, s}^\alpha f|^{2m-2}ds.
		\end{aligned}
		$$
		It follows that
		$$
		\begin{aligned}
			&\left|\mathscr{M}_{j, t}^\alpha f\right|^{2m}
			-
			\left|\mathscr{M}_{j, 1}^\alpha f\right|^{2m}\\
			&\quad=
			\int_1^t(\mathscr{M}_{j, s'}^\alpha f)^*\frac{\partial}{\partial s'}\left[\mathscr{M}_{j, s'}^\alpha f|\mathscr{M}_{j, s'}^\alpha f|^{2m-2}\right]ds'\\
			&\quad\quad+
			\int_1^t\left[\frac{\partial}{\partial s}(\mathscr{M}_{j, s}^\alpha f)^*\right]\mathscr{M}_{j, s}^\alpha f|\mathscr{M}_{j, s}^\alpha f|^{2m-2}ds\\
			&\quad=
			\int_1^t\frac{\partial}{\partial s}|\mathscr{M}_{j, s}^\alpha f|^{2m}ds,
		\end{aligned}
		$$
		which is the desired estimate.
	\end{proof}
	
	Based on the above lemma, we have the following corollary which corresponds to the celebrated Sobolev embedding theorem.
	\begin{corollary}
		Let $m$ be a positive integer. Then we have
		\begin{equation}\label{NL1}
			\left\|\sup _{1\leq t\leq2}{}^+\mathscr{M}_{j, t}^\alpha f\right\|_{L_{2m}(\mathcal{N})}^{2m}
			\leq
			\left\|\mathscr{M}_{j, 1}^\alpha f\right\|_{L_{2m}(\mathcal{N})}^{2m}
			+
			2m\left\|\mathscr{M}_{j, t}^\alpha f\right\|_{L_{2m}(\mathcal{B})}^{2m-1}
			\left\|\frac{\partial}{\partial t} \mathscr{M}_{j, t}^\alpha f\right\|_{L_{2m}(\mathcal{B})},
		\end{equation}
		where and in what follows $
		\mathcal{B}
		:=L_{\infty}\left(\mathbb{R}^n \times[1,2]\right) \overline{\otimes} \mathcal{M}.
		$
	\end{corollary}
	\begin{proof}
		By the fundamental properties of noncommutative maximal norm and (\ref{NL}), we deduce that
		$$
		\begin{aligned}
			\left\|\sup_{1\leq t\leq 2}{}^+ \mathscr{M}_{j, t}^\alpha f\right\|_{L_{2m}(\mathcal{N})}^{2m}
			\leq&
			\left\|\sup_{1\leq t\leq 2}{}^+ |\mathscr{M}_{j, t}^\alpha f|^{2m}\right\|_{L_{1}(\mathcal{N})}\\
			\leq&
			\left\|\int_{1}^{t}\frac{\partial}{\partial s} |\mathscr{M}_{j, s}^\alpha f|^{2m} ds+|\mathscr{M}_{j, 1}^\alpha f|^{2m}\right\|_{L_{1}(\mathcal{N})}\\
			\leq&
			\left\|\int_{1}^{t}\sum_{i=1}^{2m} \Psi_{j,i}f ds\right\|_{L_{1}(\mathcal{N})}
			+
			\left\|\mathscr{M}_{j, 1}^\alpha f\right\|_{L_{2m}(\mathcal{N})}^{2m},
		\end{aligned}
		$$
		where
		$$
		\Psi_{j,i}f
		=
		\begin{cases}(\mathscr{M}_{j, s}^\alpha f)^*\cdots\overset{i -th}{\overbrace{\left(\frac{\partial}{\partial s}(\mathscr{M}_{j, s}^\alpha f)^*\right)}}\mathscr{M}_{j, s}^\alpha f\cdots \overset{2m -th}{\overbrace{\mathscr{M}_{j, s}^\alpha f}}, & i\text{ is odd},\\
			(\mathscr{M}_{j, s}^\alpha f)^*\cdots(\mathscr{M}_{j, s}^\alpha f)^*\overset{i -th}{\overbrace{\left(\frac{\partial}{\partial s}\mathscr{M}_{j, s}^\alpha f\right)}}\cdots \overset{2m -th}{\overbrace{\mathscr{M}_{j, s}^\alpha f}}, & i\text{ is even}.
		\end{cases}
		$$
		Noncommutative H\"{o}lder's inequality gives that
		$$
		\begin{aligned}
			\left\|\int_{1}^{t}\sum_{i=1}^{2m} \Psi_{j,i}f ds\right\|_{L_{1}(\mathcal{N})}
			\leq&
			\left\|\int_{1}^{2}\sum_{i=1}^{2m} \Psi_{j,i}f ds\right\|_{L_{1}(\mathcal{N})}\\
			\leq&
			2m\left\|\mathscr{M}_{j, t}^\alpha f\right\|_{L_{2m}(\mathcal{B})}^{2m-1}
			\left\|\frac{\partial}{\partial t}\mathscr{M}_{j, t}^\alpha f \right\|_{L_{2m}(\mathcal{B})}.
		\end{aligned}
		$$
		This completes the proof.
	\end{proof}

	Next we show the strong-type $(p,p)$ estimates of $\{\mathscr{M}_{j,t}^\alpha\}_{1\leq t\leq 2}$ with $p=2,4,\infty$ in the following propositions. To simplify the symbols, we denote $\overline{w}:=\frac{n}{2}+\alpha-1$ throughout the rest of this section.
	
	\begin{proposition}\label{inter2}
		With all notation above, we get that for $j>0$,
		$$
		\left\|\sup _{1\leq t\leq 2}{}^+ \mathscr{M}_{j,t}^\alpha f\right\|_{L_2(\mathcal{N})}
		\lesssim
		2^{-\overline{w} j}\|f\|_{L_2(\mathcal{N})}.
		$$
	\end{proposition}
	\begin{proof}
		In view of (\ref{NL1}), we need to control
		%
		$$
		\left\|\mathscr{M}_{j, 1}^\alpha f\right\|_{L_2(\mathcal{N})}, \,\,\left\|\mathscr{M}_{j, t}^\alpha f\right\|_{L_2(\mathcal{B})} \,\,\text{  and }\,\,
		\left\|\frac{\partial}{\partial t} \mathscr{M}_{j, t}^\alpha f\right\|_{L_2(\mathcal{B})}
		.
		$$
		We first estimate the $L_2(\mathcal{B})$ norm of $\mathscr{M}_{j, t}^\alpha f$ for $1\leq t\leq 2$. It can be checked easily that
		\begin{equation}\label{L22}
			\begin{aligned}
				\left\|\mathscr{M}_{j, t}^\alpha f\right\|_{L_2(\mathcal{B})}
				=&
				2^{-(\overline{w}+\frac{1}{2}) j}\left\|\int_{\mathbb{R}^n} e^{2 \pi i \xi \cdot x} \varphi_j(t\xi) 2^{(\overline{w}+\frac{1}{2}) j} \widehat{m_\alpha}(t\xi) \widehat{f}(\xi) d \xi\right\|_{L_2(\mathcal{B})} \\
				\lesssim&
				2^{-(\overline{w}+\frac{1}{2}) j}\|f\|_{L_2(\mathcal{N})},
			\end{aligned}
		\end{equation}
		where the last inequality is a consequence of the fact that $2^{(\overline{w}+\frac{1}{2}) j} \widehat{m_\alpha}(t\xi)$ is a symbol of order zero 
		and vector-valued Plancherel's equation \eqref{plan}. Obviously, the above estimates yield
		\begin{equation}\label{L21}
			\left\|\mathscr{M}_{j, 1}^\alpha f\right\|_{L_2(\mathcal{N})}
			\lesssim
			2^{-(\overline{w}+\frac{1}{2}) j}\|f\|_{L_2(\mathcal{N})}.
		\end{equation}
		
		We now turn to estimate $\left\|\frac{\partial}{\partial t} \mathscr{M}_{j, t}^\alpha f\right\|_{L_2(\mathcal{B})}$. Note that
		$$
		\begin{aligned}
			\frac{\partial}{\partial t} \mathscr{M}_{j, t}^\alpha f(x)
			&=
			\int_{\mathbb{R}^n} e^{2 \pi i \xi \cdot x} h_j^{\alpha}(t, \xi) \widehat{f}(\xi),
		\end{aligned}
		$$
		where
		$$
		\begin{aligned}
			h_j^{\alpha}(t, \xi)
			& =\frac{\partial\widehat{m_\alpha}(t \xi)}{\partial t} \varphi_j(t \xi)+\frac{\partial\varphi_j(t \xi)}{\partial t}\widehat{m_\alpha}(t \xi).
		\end{aligned}
		$$
		From the asymptotic property of Bessel function (\ref{1.2}), we have
		$$
		\left|\frac{\partial\widehat{m_\alpha}(t \xi)}{\partial t}\right|
		\lesssim
		|t \xi|^{-(\overline{w}+\frac{1}{2})}|\xi|
		\approx
		2^j \cdot 2^{-(\overline{w}+\frac{1}{2})j}, 
		$$
		and $\left|\widehat{m_\alpha}(t \xi)\right|
		\lesssim
		|t \xi|^{-(\overline{w}+\frac{1}{2})}
		\approx
		2^{-(\overline{w}+\frac{1}{2})j}$.
		The above argument yields that
		$
		\frac{\partial}{\partial t} \mathscr{M}_{j, t}^\alpha f
		$
		behaves like
		$
		2^j \mathscr{M}_{j, t}^\alpha f.
		$
		Therefore, we have
		$$
		\left\|\frac{\partial}{\partial t} \mathscr{M}_{j, t}^\alpha f\right\|_{L_2(\mathcal{B})}
		\approx
		2^j\left\|\mathscr{M}_{j, t}^\alpha f\right\|_{L_2(\mathcal{B})}
		\lesssim
		2^j2^{-(\overline{w}+\frac{1}{2}) j}\|f\|_{L_2(\mathcal{N})} ,
		$$
		which, combined with (\ref{L22}) and (\ref{L21}), implies that
		$$
		\left\|\sup _{1\leq t\leq 2}{}^+\mathscr{M}_{j, t}^\alpha f\right\|_{L_2(\mathcal{N})}
		\lesssim
		2^{-\overline{w} j}\|f\|_{L_2(\mathcal{N})}.
		$$
		The proof is complete.
	\end{proof}

	\begin{proposition}\label{inter4}
		With all notation above, we get that for $j>0$,
		$$
		\left\|\sup _{1\leq t\leq 2}{ }^{+} \mathscr{M}_{j, t}^\alpha f\right\|_{L_4(\mathcal{N})}
		\lesssim
		2^{-\overline{w} j} 2^{(u-\frac{1}{4}) j} \|f\|_{L_4(\mathcal{N})},
		$$
		where $u$ is defined in Assumption \ref{assum}.
	\end{proposition}
	\begin{proof}
		Using the asymptotic expansion of Bessel function $\mathcal{J}_k(r)$ for $\operatorname{Re} k>-\frac{1}{2}$ (see \cite[p. 338]{S1993}), we have 
		\begin{equation}\label{Jkr}
			\mathcal{J}_k(r)
			\approx
			r^{-\frac{1}{2}}e^{i r} A_1(r)+r^{-\frac{1}{2}} e^{-i r} A_2(r),\quad r\geq 1,
		\end{equation}
		where
		$$
		A_1(r)=\sum_{\ell=0}^{\infty} c_1(\ell) r^{-\ell}, A_2(r)=\sum_{\ell=0}^{\infty} c_2(\ell) r^{-\ell}
		$$
		for some explicit coefficients $c_\sigma(\ell), \sigma=1,2$. 
		Then it follows from (\ref{1.2}) and (\ref{Jkr}) that
		$$
		\begin{aligned}
			\widehat{m_{\alpha}}(\xi)
			\approx&
			|\xi|^{-(\overline{w}+1/2)}\left( e^{2 \pi i|\xi| } A_1(2 \pi|\xi|)+ e^{-2 \pi i|\xi|} A_2(2 \pi|\xi|)
			\right) .
		\end{aligned}
		$$
		Now we can rewrite $\mathscr{M}_{j, t}^\alpha f$ as
		\begin{equation}\label{M}
			\mathscr{M}_{j, t}^\alpha f(x)
			\approx
			\mathscr{M}_{j, t}^{\alpha,1} f(x)+\mathscr{M}_{j, t}^{\alpha,2} f(x),
		\end{equation}
		where
		\begin{equation}\label{M1}
			\mathscr{M}_{j, t}^{\alpha,1} f(x)
			=
			\int_{\mathbb{R}^n} e^{2 \pi i(x \cdot \xi+t|\xi|)} \varphi_j(t \xi)|t \xi|^{-\left(\overline{w}+\frac{1}{2}\right)} A_1(2 \pi|t \xi|) \widehat{f}(\xi) d \xi
		\end{equation}
		and
		\begin{equation}\label{M2}
			\mathscr{M}_{j, t}^{\alpha,2} f(x)
			=
			\int_{\mathbb{R}^n} e^{2 \pi i(x \cdot \xi-t|\xi|)} \varphi_j(t \xi)|t \xi|^{-\left(\overline{w}+\frac{1}{2}\right)} A_2(2 \pi|t \xi|) \widehat{f}(\xi) d \xi.
		\end{equation}
		
		Choose a smooth positive function $\chi$ which is identical to one in a neighborhood of $[1,2]$ and vanishes outside the interval $(1/2,4)$. Let
		$
		\widetilde{\mathscr{M}_{j, t}^\alpha} f(x):=\chi(t) \mathscr{M}_{j, t}^\alpha f(x).
		$
		Then by 
		(\ref{NL1}), we have
		\begin{equation}\label{Mnorm}
			\begin{aligned}
				\left\|\sup _{1\leq t\leq 2}{}^+\mathscr{M}_{j, t}^\alpha f\right\|_{L_4(\mathcal{N})}^4
				\leq&
				\left\|\mathscr{M}_{j, 1}^\alpha f\right\|_{L_4(\mathcal{N})}^4
				+
				4\left\|\mathscr{M}_{j, t}^\alpha f\right\|_{L_4(\mathcal{B})}^3 \left\|\frac{\partial}{\partial t} \mathscr{M}_{j, t}^\alpha f\right\|_{L_4(\mathcal{B})}\\
				\leq&
				\left\|\mathscr{M}_{j, 1}^\alpha f\right\|_{L_4(\mathcal{N})}^4
				+4\left\|\widetilde{\mathscr{M}_{j, t}^\alpha} f\right\|_{L_4(\mathcal{A})}^3 \left\|\frac{\partial}{\partial t} \widetilde{\mathscr{M}_{j, t}^\alpha} f\right\|_{L_4(\mathcal{A})},
			\end{aligned}
		\end{equation}
		where $\mathcal{A}:=L_{\infty}\left(\mathbb{R}^n \times \mathbb{R}\right) \overline{\otimes} \mathcal{M}$. This enables us to reduce matters to the estimates of
		$\left\|\mathscr{M}_{j, 1}^\alpha f\right\|_{L_4(\mathcal{N})}$,
		$\left\|\widetilde{\mathscr{M}_{j, t}^\alpha} f\right\|_{L_4(\mathcal{A})}$ and $\left\|\frac{\partial}{\partial t} \widetilde{\mathscr{M}_{j, t}^\alpha} f\right\|_{L_4(\mathcal{A})}$.
		

		We first estimate the $L_4(\mathcal{A})$ norm of $\widetilde{\mathscr{M}_{j, t}^\alpha} f$.
		Applying Assumption \ref{assum} to $\chi(t)\mathscr{M}_{j, t}^{\alpha,1} f$, we have
		$$
		\begin{aligned}
			\left\|\chi(t)\mathscr{M}_{j, t}^{\alpha,1} f\right\|_{L_4(\mathcal{A})} 
			\lesssim
			2^{-\left(\overline{w}+\frac{1}{2}\right)j} 2^{u j}\|f\|_{L_4(\mathcal{N})},
		\end{aligned}
		$$
		since $2^{\left(\overline{w}+\frac{1}{2}\right) j}|t \xi|^{-\left(\overline{w}+\frac{1}{2}\right)} A_1(2 \pi|t \xi|)$ is a symbol of order zero.
		Similar inequality can also be obtained for $\chi(t)\mathscr{M}_{j, t}^{\alpha,2} f$, then clearly
		\begin{equation}\label{Mnorm1}
			\left\|\widetilde{\mathscr{M}_{j, t}^\alpha} f\right\|_{L_4(\mathcal{A})}
			\lesssim
			2^{-\left(\overline{w}+\frac{1}{2}\right) j} 2^{u j}\|f\|_{L_4(\mathcal{N})} .
		\end{equation}
		As a special case of the above estimate, we also have
		\begin{equation}\label{Mnorm11}
			\left\|\mathscr{M}_{j, 1}^\alpha f\right\|_{L_4(\mathcal{N})}
			\lesssim
			2^{-\left(\overline{w}+\frac{1}{2}\right) j} 2^{u j}\|f\|_{L_4(\mathcal{N})} .
		\end{equation}

		Next we consider the terms related to $\frac{\partial}{\partial t} \widetilde{\mathscr{M}_{j, t}^{\alpha}} f$. An easy argument shows that
		$$
		\begin{aligned}
			\frac{\partial}{\partial t} (\chi(t)\mathscr{M}_{j, t}^{\alpha,1} f(x))
			= &
			\int_{\mathbb{R}^n} e^{2 \pi i(x \cdot \xi+t|\xi|)} k_j(t, \xi) \widehat{f}(\xi) d \xi,
		\end{aligned}
		$$
		where
		$$
		\begin{aligned}
			k_j(t, \xi)
			=&
			\left(\frac{\partial\chi(t)}{\partial t} \varphi_j(t \xi)
			+
			\chi(t) \frac{\partial\varphi_j(t \xi)}{\partial t}\right)
			|t \xi|^{-\left(\overline{w}+\frac{1}{2}\right)} A_1(2 \pi|t \xi|) \\
			&+
			\left[2\pi i-
			\left(\overline{w}+\frac{1}{2}\right)|t\xi|^{-1} \right]|t \xi|^{-\left(\overline{w}+\frac{1}{2}\right)}
			\chi(t)|\xi| \varphi_j(t \xi)A_1(2 \pi|t \xi|) \\
			&+
			\chi(t) \varphi_j(t \xi)|t \xi|^{-\left(\overline{w}+\frac{1}{2}\right)} \frac{\partial A_1(2 \pi|t \xi|)}{\partial t}\\
			\lesssim&
			\frac{\partial\varphi_j(t \xi)}{\partial t}|t \xi|^{-\left(\overline{w}+\frac{1}{2}\right)} A_1(2 \pi|t \xi|)
			+
			\frac{\partial\chi(t)}{\partial t} \varphi_j(t \xi)|t \xi|^{-\left(\overline{w}+\frac{1}{2}\right)} A_1(2 \pi|t \xi|) \\
			& +\varphi_j(t \xi)|t \xi|^{-\left(\overline{w}-\frac{1}{2}\right)} B(2 \pi|t \xi|),
		\end{aligned}
		$$
		and $B(r)$ is a symbol of order zero.
		Similar analysis is valid for $\frac{\partial}{\partial t} (\chi(t)\mathscr{M}_{j, t}^{\alpha,2} f(x))$. Based on the above argument, we know that $\frac{\partial}{\partial t} \widetilde{\mathscr{M}_{j, t}^\alpha} f$ behaves like $2^j \widetilde{\mathscr{M}_{j, t}^\alpha} f$. So we have
		\begin{equation}\label{Mnorm2}
			\left\|\frac{\partial}{\partial t} \widetilde{\mathscr{M}_{j, t}^\alpha} f\right\|_{L_4(\mathcal{A})}
			\lesssim
			2^{-\left(\overline{w}+\frac{1}{2}\right) j} 2^j 2^{u j}\|f\|_{L_4(\mathcal{N})} .
		\end{equation}

		By (\ref{Mnorm}), 
		(\ref{Mnorm1}), (\ref{Mnorm11}) and (\ref{Mnorm2}), we conclude that
		$$
		\left\|\sup _{1\leq t\leq 2}{ }^{+} \mathscr{M}_{j, t}^\alpha f\right\|_{L_4(\mathcal{N})}
		\lesssim
		2^{-\overline{w} j} 2^{(u-\frac{1}{4}) j}\|f\|_{L_4(\mathcal{N})} .
		$$
		This completes the proof.
	\end{proof}

	\begin{proposition}\label{interinfty}
		With all notation above, we get that for $j>0$,
		$$
		\left\|\sup _{1\leq t\leq 2}{}^+ \mathscr{M}_{j,t}^\alpha f\right\|_{L_{\infty}(\mathcal{N})}
		\lesssim
		2^{-\overline{w}j}\|f\|_{L_{\infty}(\mathcal{N})}.
		$$
	\end{proposition}
	\begin{proof}Define
		$$
		G_{j,t}^{\alpha}(x)
		=
		\int_{\mathbb{R}^n} e^{2 \pi i x \cdot \xi} 2^{\overline{w} j} \widehat{m_\alpha}(t \xi) \varphi_j(t \xi) d \xi,\quad
		x\in \mathbb{R}^n,\quad t>0.
		$$
		It can be deduced from the fact $L_{\infty}(\mathcal{M};\ell_{\infty})=\ell_{\infty}(L_{\infty}(\mathcal{M}))$ and vector-valued Young's inequality that 
		$$
		\begin{aligned}
			\left\|\sup_{1\leq t\leq 2}{}^+\mathscr{M}_{j, t}^\alpha f\right\|_{L_{\infty}(\mathcal{N})}
			=&
			2^{-\overline{w} j}\sup _{1\leq t\leq 2}\left\|G_{j,t}^{\alpha} * f\right\|_{L_{\infty}(\mathcal{N})} \\
			\lesssim&
			2^{-\overline{w}j} \sup _{1\leq t\leq 2}\left\|G_{j,t}^{\alpha}\right\|_{L_1\left(\mathbb{R}^n\right)}\|f\|_{L_{\infty}(\mathcal{N})}.
		\end{aligned}
		$$
		Hence, the required conclusion easily follows from
		\begin{equation}\label{Gt}
			\left\|\chi(t) G_{j,t}^{\alpha}\right\|_{L_1\left(\mathbb{R}^{n}\right)} \lesssim 1, \quad t>0,
		\end{equation}
		where $\chi(t)$ is as in Proposition \ref{inter4}.
		The proof of (\ref{Gt}) was given 
		in \cite[p. 25]{L15} for the special case  $n=2$, 
		which, in fact, applies 
		to higher dimensions as well.
	\end{proof}
	
	Now we are in the position to complete the proof of Proposition \ref{lemma12}.
	
	\begin{proof}[Proof of Proposition \ref{lemma12}]
		
		We employ the noncommutative interpolation theorem to prove this proposition.
		By Proposition \ref{inter2}, Proposition \ref{inter4} and Lemma \ref{interpolation}, we have for $2\leq p\leq 4$,
		$$
		\begin{aligned}
			\left\|\sup_{1\leq t\leq2}{}^+ \mathscr{M}_{j,t}^\alpha f\right\|_{L_{p}(\mathcal{N})}
			\lesssim&
			2^{-\overline{w} j(1-\theta_1)}\left[2^{-\overline{w} j} 2^{(u-\frac{1}{4}) j}\right]^{\theta_1}\left\| f\right\|_{L_p(\mathcal{N})}\\
			=&
			2^{-\overline{w} j}2^{(u-\frac{1}{4}) \theta_1 j}\left\| f\right\|_{L_p(\mathcal{N})},
		\end{aligned}
		$$
		where $\theta_1=2-4/p$. Let $\mu=-\overline{w}+(u-\frac{1}{4}) \theta_1$, by (\ref{alpha}), it is easy to see that $\mu<0$.

		Similarly, applying Proposition \ref{inter4}, Proposition \ref{interinfty} and Lemma \ref{interpolation}, we obtain for $4\leq p\leq \infty$,
		$$
		\begin{aligned}
			\left\|\sup_{1\leq t\leq2}{}^+ \mathscr{M}_{j,t}^\alpha f\right\|_{L_{p}(\mathcal{N})}
			\lesssim&
			\left[2^{-\overline{w} j} 2^{(u-\frac{1}{4}) j} \right]^{1-\theta_2}2^{-\overline{w}\theta_2}\left\| f\right\|_{L_p(\mathcal{N})}\\
			=&
			2^{-\overline{w} j}
			2^{\left(u-\frac{1}{4}\right)(1-\theta_2) j}\left\| f\right\|_{L_p(\mathcal{N})},
		\end{aligned}
		$$
		where $\theta_2=1-4/p .$ Let $\mu=-\overline{w}+\left(u-\frac{1}{4}\right)(1-\theta_2)$, by (\ref{alpha}), it is easy to see that $\mu<0$.
	\end{proof}
	
	\begin{remark}
		According to Assumption \ref{assum}, the sharp results can be obtained by the interpolation at $p=2$, $\bar{p}_n$ and $\infty$, 
		where $2< \bar{p}_n \leq 4$ for $n \geq 2$. Therefore our results are sharp for $n=2$ and $n=3$ since $\bar{p}_2=\bar{p}_3=4$.
		
		
	\end{remark}


\begin{thebibliography}{10}
		\bibitem{B1986}
		J. Bourgain, Averages in the plane over convex curves and maximal operators, J. Anal. Math. \textbf{47} (1986), 69-85.
		
		
		\bibitem{BD2015}
		J. Bourgain, C. Demeter, The proof of the $\ell^2$ decoupling conjecture, Ann. of Math. (2) \textbf{182} (2015), no. 1, 351-389.
		
		\bibitem{CXY2013}
		Z. Chen, Q. Xu, Z. Yin, Harmonic analysis on quantum tori, Comm. Math. Phys. \textbf{322} (2013), no. 3, 755-805.
		
		\bibitem{FK1986}T. Fack, H. Kosaki, Generalized $s$-numbers of $\tau$-measurable operators, Pacific J. Math. \textbf{123} (1986), no. 2, 269-300.
		
		\bibitem{GJP2017}
		A. Gonz\'{a}lez-P\'{e}  rez, M. Junge, J. Parcet, Smooth Fourier multipliers in group algebras via Sobolev dimension, Ann. Sci. \'{E}c. Norm. Sup\'{e}r. \textbf{50} (2017), no. 4, 879-925.
		
		\bibitem{GJP}
		A. Gonz\'{a}lez-P\'{e}rez, M. Junge, J. Parcet, Singular integrals in quantum euclidean spaces, Mem. Amer. Math. Soc. \textbf{272}(2021), no. 1334, xiii+90pp.
		
		\bibitem{GLMX2023}
		C. Gao, B. Liu, C. Miao, Y. Xi, Square funcion estimates and local smoothing for fourier integral operators. Proc. Lond. Math. Soc. \textbf{126} (2023), no. 6, 1923-1960.
		
		\bibitem{GS2010}
		G. Garrigs, A. Seeger. A mixed norm variant of Wolff's inequality for paraboloids, In Harmonic analysis
		and partial differential equations, volume 505 of Contemp. Math. pp. 179-197. American Mathematical
		Society, Providence (2010).
		
		\bibitem{GWZ2020}
		L. Guth, H. Wang, R. Zhang, A sharp square function estimate for the cone in $\mathbb{R}^3$,
		Ann. of Math (2) \textbf{192} (2020), no. 2, 551-581.
		
		\bibitem{HLX2022}
		G. Hong, X. Lai, B. Xu, Maximal singular integral operators acting on noncommutative $L_p$-spaces, Math. Ann. \textbf{386} (2023), no. 1-2, 375-414.
		
		\bibitem{HLW}
		G. Hong, X. Lai, L. Wang, Public talks.
		
		\bibitem{HWW}
		G. Hong, S. Wang, X. Wang, Pointwise convergence of noncommutative Fourier series, arXiv:1908.00240.
		
		\bibitem{J2002}
		M. Junge. Doob's inequality for noncommutative martingale, J. Reine Angew. Math. \textbf{549} (2002), 149-190.
		
		\bibitem{JLX2006}
		M. Junge, C. Le Merdy, Q. Xu, $H^{\infty}$-functional calculus and square functions on noncommutative
		$L_p$-spaces, Ast\'{e}risque. \textbf{305} (2006), vi+138pp.
		
		\bibitem{JM2010}
		M. Junge, T. Mei, Noncommutative Riesz transforms-a probabilistic approach, Amer. J. Math.
		\textbf{132} (2010), no. 3, 611-680.
		
		\bibitem{JM2012}
		M. Junge, T. Mei, BMO spaces associated with semigroups of operators, Math. Ann. \textbf{352} (2012), no. 3, 691-743.
		
		\bibitem{JMP2014}
		M. Junge, T. Mei, J. Parcet, Smooth Fourier multipliers on group von Neumann algebras, Geom.
		Funct. Anal. \textbf{24} (2014), no. 6, 1913-1980.
		
		\bibitem{JMP2018}
		M. Junge, T. Mei, J. Parcet, Noncommutative Riesz transforms-dimension free bounds and Fourier
		multipliers, J. Eur. Math. Soc. \textbf{20} (2018), no. 3, 529-595.
		
		\bibitem{JMPX2021}
		M. Junge, T. Mei, J. Parcet, R, Xia, Algebraic Calder\'{o}n-Zygmund theory, Adv. Math. \textbf{376} (2021), 72 pp.
		
		\bibitem{JX2007}
		M. Junge, Q. Xu, Noncommutative maximal ergodic theorems, J. Amer. Math. Soc. \textbf{20} (2007), no. 2, 385-439.
		
		\bibitem{L1976}
		E. C. Lance. Ergodic theorems for convex sets and operator algebras, Invent Math. \textbf{37} (1976), no.3, 201-214.
		
		\bibitem{L15}
		W. Li. Maximal functions associated with nonisotropic dilations of hypersurfaces in $\mathbb{R}^3$. Thesis. 2015.
		
		
		\bibitem{LW2002}
		I. Laba, T. Wolff. A local smoothing estimate in higher dimensions. J. Anal. Math. \textbf{88}, 2002, 149-171.
		
		\bibitem{LSZ2020}
		G. Levitina, F. Sukochev, D. Zanin, Cwikel estimates revisited, Proc. Lond. Math. Soc. \textbf{120} (2020), no. 2, 265-304.
		
		\bibitem{LSSY2023}
		N. Liu, M. Shen, L. Song, L. Yan, $L^p$ bounds for Stein's spherical maximal operators, arXiv:2303.08655.
		
		\bibitem{MM2013}
		A. Martini, D. M\"{u}ller, Spectral multiplier theorems of Euclidean type on further classes of two-step
		stratified groups, Proc. Lond. Math. Soc (3) \textbf{109} (2014), no. 5, 1229-1263.
		
		\bibitem{M2007}
		T. Mei, Operator valued Hardy spaces, Mem. Amer. Math. Soc. \textbf{188} (2007), no. 881, vi+64 pp.
		
		\bibitem{MP2009}
		T. Mei, J. Parcet, Pseudo-localization of singular integrals and noncommutative Littlewood-Paley
		inequalities, Int. Math. Res. Not. \textbf{8} (2009), 1433-1487.
		
		
		
		\bibitem{MSX2020}
		E. McDonald, F. Sukochev, X. Xiong, Quantum differentiability on noncommutative Euclidean
		spaces, Comm. Math. Phys. \textbf{379} (2020), no. 2, 491-542.
		
		\bibitem{MYZ2017}
		C. Miao, J. Yang, J. Zheng, On local smoothing problems and Stein's maximal spherical means,
		Proc. Amer. Math. Soc. \textbf{145} (2017), no. 10, 4269-4282.
		
		\bibitem{MSS1992}
		G. Mockenhaupt, A. Seeger, C.D. Sogge. Wave front sets, local smoothing and Bourgain's circular
		maximal theorem, Ann. of Math. (2) \textbf{136} (1992), no. 1, 207-218.
		
		
		\bibitem{PRS}
		J. Parcet, \'{E}. Ricard, M. de la salle. Fourier multipliers in $SL_n(R)$, Duke Math. J. \textbf{171} (2022), no. 6, 1235-1297.
		
		\bibitem{P1998}
		G. Pisier, Noncommutative vector-valued $L_p$-spaces and completely $p$-summing maps, Ast\'{e}risque. \textbf{247} (1998), vi+131 pp.
		
		\bibitem{PX1997}
		G. Pisier, Q. Xu, Non-commutative martingale inequalities, Comm. Math. Phys. \textbf{189} (1997), no.3, 667-698.
		
		\bibitem{PX2003}
		G. Pisier, Q. Xu, Noncommutative $L_p$ spaces. Handbook of geometry of Banach spaces, 2003, 1459-1517.
		
		\bibitem{S1991}
		C. Sogge, Propagation of singularities and maximal functions in the plane, Invent. Math. \textbf{104} (1991), no. 2, 349-376.
		
		\bibitem{S2017}
		C. Sogge, Fourier integrals in classical analysis. Second edition. Cambridge Tracts in Mathematics, 210.
		Cambridge University Press, Cambridge, 2017. 
		
		\bibitem{S1976}
		E.M. Stein. Maximal functions. I. Spherical means, Proc. Nat. Acad. Sci. U.S.A. \textbf{73} (1976), no. 7, 2174-2175.
		
		\bibitem{S1993}
		E.M. Stein. Harmonic Analysis: Real-variable Mehtods, Orthogonality and Oscillatory Integrals. Princeton Mathematical Series, vol. 43, Princeton University Press, Princeton, NJ, 1993.
		
		
		\bibitem{SW1971}
		E.M. Stein, G. Weiss. Introduction to Fourier analysis on Euclidean spaces,
		Princeton University Press, Princeton, NJ, 1971. Princeton Mathematical Series, No. 32.
		
		\bibitem{W1992}
		G.N. Watson, A treatise on the theory of Bessel functions, Cambridge University Press, Cambridge, 1995. 
		
		\bibitem{W2000}
		T. Wolff. Local smoothing type estimates on $L^p$ for large $p$, Geom. Funct. Anal. \textbf{10} (2000), no. 5, 1237-1288.
		
		\bibitem{XX2018}
		R. Xia, X. Xiong, Operator-valued Triebel-Lizorkin spaces, Integral Equ. Oper. Theory, \textbf{90} (2018), no. 6, 65pp.
		
		\bibitem{XXX2016}
		R. Xia, X. Xiong, Q. Xu, Characterizations of operator-valued Hardy spaces and applications to
		harmonic analysis on quantum tori, Adv. Math. \textbf{291} (2016), 183-227.
		
		\bibitem{XXY2018}
		X. Xiong, Q. Xu, Z. Yin, Sobolev, Besov and Triebel-Lizorkin spaces on quantum tori, Mem.
		Amer. Math. Soc. \textbf{252} (2018), vi+118 pp.
		
		\bibitem{Y1977}
		F.J. Yeadon, Ergodic theorems for semifinite von Neumann algebras. I, J. London Math. Soc. (2) \textbf{16}
		(1977), no. 2, 326-332.
	\end{thebibliography}
\end{document}